\documentclass{amsart}

\usepackage{amssymb,amsfonts,amsthm,amsmath,amscd}
\usepackage[all,arc]{xy}
\usepackage{enumerate}
\usepackage{mathrsfs}
\usepackage{color}
\usepackage{hyperref}
\usepackage{fullpage}

\newtheorem{thm}{Theorem}[section]

\newtheorem{prop}[thm]{Proposition}
\newtheorem{lem}[thm]{Lemma}

\theoremstyle{definition}

\theoremstyle{remark}
\newtheorem{rem}[thm]{Remark}

\DeclareMathOperator{\stab}{stab}
\DeclareMathOperator{\cochar}{cochar}
\DeclareMathOperator{\Leaf}{Leaf}
\DeclareMathOperator{\Slope}{Slope}
\DeclareMathOperator{\supp}{supp}

\DeclareMathOperator{\Spec}{Spec}
\DeclareMathOperator{\Aut}{Aut}
\DeclareMathOperator{\sym}{sym}
\DeclareMathOperator{\vir}{vir}

\makeatletter
\let\c@equation\c@thm
\makeatother
\numberwithin{equation}{section}

\title{Equivariant quantum cohomology of cotangent bundle of $G/P$}

\author{Changjian SU}
\address{Department of Mathematics\\
  Columbia University\\
  New York, NY 10027}
\email{changjian@math.columbia.edu}

\begin{document}

\begin{abstract}

Let $G$ denote a complex semisimple linear algebraic group, $P$ a parabolic subgroup of $G$ and $\mathcal{P}=G/P$. We identify the quantum multiplication by divisors in $T^*\mathcal{P}$ in terms of stable basis, which is introduced in \cite{Maulik2012}. Using this and the restriction formula for stable basis (\cite{su2015restriction}), we show that the $G\times\mathbb{C}^*$-equivariant quantum multiplication formula in $T^*\mathcal{P}$ is conjugate to the formula conjectured by Braverman.

\end{abstract}

\maketitle

\section{Introduction}

The main goal of this paper is to study the equivariant quantum cohomology of $T^*\mathcal{P}$, which is a special case of symplectic resolutions. Recall from \cite{Kaledin2009} that a smooth algebraic variety $X$ with a holomorphic symplectic form $\omega$ is called a symplectic resolution if the affinization map 
\[X\rightarrow X_0=\Spec H^0(X,\mathcal{O}_X)\]
is projective and birational. Conjecturally all the symplectic resolutions of the form $T^*M$ for a smooth algebraic variety  $M$ are of the form $T^*\mathcal{P}$, see \cite{Kaledin2009}. In \cite{Fu}, Fu proved that every symplectic resolution of a normalization of a nilpotent orbit closure in a semisimple Lie algebra $\mathfrak{g}$ is isomorphic to $T^*\mathcal{P}$ for some parabolic subgroup $P$ in $G$. 

In \cite{Maulik2012}, Maulik and Okounkov defined the stable basis for a wide class of varieties, which include symplectic resolutions. Other examples of symplectic resolutions include hypertoric varieties, resolutions of Slodowy slices, Hilbert schemes of points on $\mathbb{C}^2$, and, more generally, Nakajima varieties \cite{Nakajima1998}. Their quantum cohomologies were studied in \cite{McBreen2013}, \cite{Braverman2011}, \cite{Okounkov2004} and \cite{Maulik2012} respectively. The stable basis in the Springer resolutions are just characteristic cycles of Verma modules up to a sign,  see \cite{ginsburg1986} and Remark 3.5.3 in \cite{Maulik2012}, and the restriction of stable basis to fixed points is obtained in \cite{su2015restriction}. In the case of Hilbert schemes of points on $\mathbb{C}^2$, it corresponds to Schur functions if we identify the equivariant cohomology ring of Hilbert schemes with the symmetric functions, while the fixed point basis corresponds to Jack symmetric functions, see e.g. \cite{Maulik2012}, \cite{hiraku1999lectures}, \cite{nakajima2014more}. In this case, Shenfeld obtained the transition matrix from the stable basis to fixed point basis in \cite{Shenfeld}.

To state our main Theorem, let us fix some notations. Let $B$ be a Borel subgroup, $R^+$ be the roots appearing in $B$, and $R^-=-R^+$. Let $\Delta$ be the set of simple roots, $I$ be a subset of $\Delta$, and $P=P_I=\bigcup_{w\in W_I}BwB$ be the parabolic subgroup containing $B$ corresponding to $I$. It is well-known that every parabolic subgroup is conjugate to some parabolic subgroup containing the fixed Borel subgroup $B$, which is of the form $P_I$ for some subset $I$ in $\Delta$, and  $P_I$ is not conjugate to $P_J$ if the two subsets $I$ and $J$ are not equal (see \cite{Springer2010}). Let $W_P$ the subgroup of the Weyl group $W$ generated by the simple reflections $\sigma_\alpha$ for $\alpha\in I$, and $R_P^{\pm}$ be the roots in $R^{\pm}$ spanned by $I$. Let $\alpha^{\vee}$ be the coroot corresponding to $\alpha$. Let $A$ be a maximal torus of $G$ contained in $B$, and \quad $\mathbb{C}^*$ scales the fiber of $T^*\mathcal{P}$ by a nontrivial character $-\hbar$. Let $T=A\times \mathbb{C}^*$.

Any weight $\lambda$ that vanishes on all $\alpha^{\vee}\in I^{\vee}$ determines a one-dimensional representation $\mathbb{C}_{\lambda}$ of $P$. Define a line bundle
\[\mathcal{L}_{\lambda}=G\times_P\mathbb{C}_{\lambda}\]
on $G/P$. Pulling it back to $T^*\mathcal{P}$, we get a line bundle on $T^*\mathcal{P}$, which will still be denoted by $\mathcal{L}_{\lambda}$. Let $D_{\lambda}:=c_1(\mathcal{L}_{\lambda})$. It is well-known that the fixed point set $(T^*\mathcal{P})^A$ is in one-to-one correspondence with $W/W_P$. The stable envelope map $\stab_+$ will be defined in Section \ref{section 2}, and $\stab_+({\bar{y}})$ is the image of the unit in $H_T^*(\bar{y})$ under the stable envelope map, where $\bar{y}$ in $H_T^*(\bar{y})$ is the fixed point in $T^*\mathcal{P}$ corresponding to $yW_P$. An element $y\in W$ is called minimal if its length is minimal among the elements in the coset $yW_P$. As $y$ runs through the minimal elements, $\stab_+({\bar{y}})$ form a basis in $H_T^*(T^*\mathcal{P})$ after localization, which is called the stable basis. The result we are going to prove is:
\begin{thm}\label{quantum mul for P}
The quantum multiplication by $D_{\lambda}$ in $H_T^*(T^*\mathcal{P})$ is given by:
\begin{align*}
D_{\lambda}\ast \stab_+({\bar{y}})=& y(\lambda)\stab_+({\bar{y}})-\hbar\sum\limits_{\alpha\in R^+, y\alpha\in R^-}(\lambda,\alpha^{\vee})\stab_+(\overline{y\sigma_{\alpha}})\\
&-\hbar\sum_{\alpha\in R^+\setminus R^+_P}(\lambda, \alpha^{\vee})\frac{q^{d(\alpha)}}{1-q^{d(\alpha)}} \left(\stab_+(\overline{y\sigma_{\alpha}})+\prod\limits_{\beta\in R_P^+}\frac{\sigma_{\alpha}\beta}{\beta} \stab_+(\bar{y})\right),
\end{align*}
where $y$ is a minimal representative in $yW_P$, and $d(\alpha)$ is defined by Equation \ref{degree}.
\end{thm}

Combining this and the restriction formula for stable basis (\cite{su2015restriction}), we get
\begin{thm}\label{P con}
Under the isomorphism $H_{G\times\mathbb{C}^*}^*(T^*\mathcal{P})\simeq (\sym\mathfrak{t}^*)^{W_P}[\hbar]$, the operator of quantum multiplication by $D_{\lambda}$ is given by
\[D_{\lambda}\ast f=\lambda f+\hbar\sum_{\alpha\in R^+\setminus R^+_P}(\lambda, \alpha^{\vee})\frac{q^{d(\alpha)}}{1-q^{d(\alpha)}}\left(\frac{\tilde{\sigma}_{\alpha}(f\prod\limits_{\beta\in R^+_P}(\beta-\hbar))}{\prod\limits_{\beta\in R^+_P}(\beta-\hbar)}-\frac{\prod\limits_{\beta\in R^+_P}\sigma_{\alpha}\beta}{\prod\limits_{\beta\in R^+_P}\beta}f\right).\]
\end{thm}
This shows that the quantum multiplication formula is conjugate to the one (\ref{conj}) conjectured (through private communication) by Professor Braverman.

The paper is organized as follows. In Section 2, we apply results in \cite{Maulik2012} to define the stable basis of $T^*\mathcal{P}$. In Section 3, we prove our main Theorem \ref{quantum mul for P} by calculating the classical multiplication and purely quantum multiplication separately. In the last section, we first show how to deduce the $G\times\mathbb{C}^*$-equivariant quantum multiplication in $T^*(G/B)$ from Theorem \ref{quantum mul for P}, which is the main result of \cite{Braverman2011}. Then a similar calculation gives a proof to Theorem \ref{P con}. 

\subsection*{Acknowledgments}  I wish to express my deepest thanks to my advisor Professor Andrei Okounkov for suggesting this problem to me and his endless help, patience and invaluable guidance. I am grateful to Professor Alexander Braverman for suggesting the conjectured formula (\ref{conj}) to me. I also thank Chiu-Chu Liu, Michael McBreen, Davesh Maulik, Andrei Negut, Andrey Smirnov, Zijun Zhou, Zhengyu Zong for many stimulating conversations and emails. A lot of thanks also go to my friend Pak-Hin Lee for editing a previous version of the paper.

\section{Stable basis for $T^*\mathcal{P}$}
\label{section 2}
In this section, we apply the construction in \cite{Maulik2012} to $T^*\mathcal{P}$.

\subsection{Fixed point sets}
It is well-known the $A$-fixed points of $T^*\mathcal{P}$ is in one-to-one correspondence with $W/W_P$. For any $y\in W$, let $\bar{y}$ denote the coset $yW_P$ and the corresponding fixed point in $T^*\mathcal{P}$. Recall the Bruhat order $\leq$ on $W/W_P$ is defined as follows:
\[\bar{y}\leq \bar{w}\text{\quad if\quad} ByP/P\subseteq \overline{BwP/P}.\]

\subsection{Chamber decomposition}
The cocharacters
\[
\sigma:\mathbb{C^*}\rightarrow A
\]
form a lattice. Let 
\[\mathfrak{a}_{\mathbb{R}}=\cochar(A)\otimes_{\mathbb{Z}}\mathbb{R}.\]

Define the torus roots to be the $A$-weights occurring in the normal bundle to $(T^*\mathcal{P})^A$. Then the root hyperplanes partition $\mathfrak{a}_{\mathbb{R}}$ into finitely many chambers
\[\mathfrak{a}_{\mathbb{R}}\setminus\bigcup \alpha_i^\perp=\coprod \mathfrak{C}_i.\]
It is easy to see that in this case the torus roots are just the roots in $G$. Let $+$ denote the chamber such that all root in $R^+$ are positive on it, and $-$ the opposite chamber.

\subsection{Stable leaves}
Let $\mathfrak{C}$ be a chamber. Define the stable leaf of $\bar{y}$ by
\[\Leaf_{\mathfrak{C}}(\bar{y})=\left\{x\in T^*\mathcal{P}\left|\lim\limits_{z\rightarrow 0} \sigma(z)\cdot x=\bar{y}\right.\right\},\]
where $\sigma$ is any cocharacter in $\mathfrak{C}$; the limit is independent of the choice of $\sigma\in \mathfrak{C}$. In our case, 
\[\Leaf_+(\bar{y})=T_{B\bar{y}P/P}^*\mathcal{P},\] and \[\Leaf_-(\bar{y})=T_{B^-\bar{y}P/P}^*\mathcal{P},\] where $B^-$ is the opposite Borel subgroup.

Define a partial order on $W/W_P$ as follows:
\[\bar{w}\preceq_{\mathfrak{C}} \bar{y}\text{\quad if\quad}\overline{\Leaf_{\mathfrak{C}}(\bar{y})}\cap \bar{w}\neq \emptyset.\]
By the description of $\Leaf_+(\bar{y})$, the order $\preceq_+$ is the same as the Bruhat order on $W/W_P$, and $\preceq_-$ is the opposite order. Define the slope of a fixed point $\bar{y}$ by
\[\Slope_{\mathfrak{C}}(\bar{y})=\bigcup_{\bar{w}\preceq_{\mathfrak{C}} \bar{y}} \Leaf_{\mathfrak{C}}(\bar{w}).\] 

\subsection{Stable basis}
For each $\bar{y}$, define $\epsilon_{\bar{y}}=e^A(T_{\bar{y}}^*\mathcal{P})$. Here, $e^A$ denotes the $A$-equivariant Euler class. Let $N_{\bar{y}}$ denote the normal bundle of $T^*\mathcal{P}$ at the fixed point $\bar{y}$. The chamber $\mathfrak{C}$ gives a decomposition of the normal bundle 
\[
N_{\bar{y}}=N_{\bar{y},+}\oplus N_{\bar{y},-}\]
into $A$-weights which are positive and negative on $\mathfrak{C}$ respectively. The sign in $\pm e(N_{\bar{y},-})$ is determined by the condition
\[
\pm e(N_{\bar{y},-})|_{H_A^*(\text{pt})}=\epsilon_{\bar{y}}.\]
The following theorem is the Theorem 3.3.4 in \cite{Maulik2012} applied to $T^*\mathcal{P}$.
 
\begin{thm}[\cite{Maulik2012}]\label{stable for P}
There exists a unique map of $H_T^*(\text{pt})$-modules
\begin{center}
$\stab_{\mathfrak{C}}:H_T^*((T^*\mathcal{P})^A)\rightarrow H_T^*(T^*\mathcal{P})$
\end{center}
such that for any $\bar{y}\in W/W_P$, $\Gamma=\stab_{\mathfrak{C}}(\bar{y})$ satisfies:
\begin{enumerate}
\item $\supp\Gamma\subset \Slope_{\mathfrak{C}}(\bar{y})$,
\item $\Gamma|_{\bar{y}}=\pm e(N_{-,\bar{y}})$, with sign according to $\epsilon_{\bar{y}}$,
\item $\Gamma|_{\bar{w}}$ is divisible by $\hbar$, for any $\bar{w}\prec_{\mathfrak{C}} \bar{y}$,
\end{enumerate}
where $\bar{y}$ in $\stab_{\mathfrak{C}}(\bar{y})$ denotes the unit in $H_T^*(\bar{y})$.
\end{thm}

\begin{rem}
\leavevmode
\begin{enumerate}
\item
The map is defined by a Lagrangian correspondence between $(T^*\mathcal{P})^A\times T^*\mathcal{P}$, hence maps middle degree to middle degree.
\item 
From the characterization, the transition matrix from $\{\stab_{\mathfrak{C}}(\bar{y}), \bar{y}\in W/W_P\}$ to the fixed point basis is a triangular matrix with nontrivial diagonal terms. Hence, after localization, $\{\stab_{\mathfrak{C}}(\bar{y}), \bar{y}\in W/W_P\}$ form a basis for the cohomology, which is the \textbf{stable basis}.
\item
Theorem 4.4.1 in \cite{Maulik2012} shows that
$\{\stab_{\mathfrak{C}}(\bar{y}), \bar{y}\in W/W_P\}$ and $\{(-1)^m\stab_{\mathfrak{-C}}(\bar{y}), \bar{y}\in W/W_P\}$
are dual bases, where $m=\dim G/P$.
\end{enumerate}
\end{rem}
From now on, we let $\stab_\pm(\bar{y})$ denote the stable basis in $H_T^*(T^*\mathcal{P})$, and let $\stab_\pm(y)$ denote the stable basis in $H_T^*(T^*\mathcal{B})$. We record two lemmas here, which will be important for the calculations.
\begin{lem}[\cite{Bernstein1973a}]
Each coset $W/W_P$ contains exactly one element of minimal length, which is characterized by the property that it maps $I$ into $R^+$.
\end{lem}
\begin{lem}[\cite{su2015restriction}]\label{mod h^2 P}
Let $y$ be a minimal representative of the coset $yW_P$. Then
\[\stab_+(\bar{y})|_{\bar{w}} \equiv \left\{\begin{array}{ccc}
\displaystyle (-1)^{l(y)+1}\frac{\hbar\prod\limits_{\alpha\in R^+}\alpha}{y\beta\prod\limits_{\alpha\in R^+_P} y\sigma_{\beta}\alpha} & \pmod{\hbar^2} & \text{if } \bar{w}=\overline{y\sigma_{\beta}} \text{ and } y\sigma_{\beta}<y \text{ for some } \beta\in R^+,\\
\\
0 & \pmod{\hbar^2} & \text{otherwise},
\end{array}\right.\]
and
\[\stab_-(\bar{w})|_{\bar{y}}\equiv\left\{\begin{array}{ccc}
\displaystyle (-1)^{l(y)+1}\frac{\hbar\prod\limits_{\alpha\in R^+}\alpha}{y\beta\prod\limits_{\alpha\in R^+_P} y\alpha} &\pmod{\hbar^2} &\text{if } \bar{w}=\overline{y\sigma_{\beta}} \text{ and } y\sigma_{\beta}<y \text{ for some } \beta\in R^+,\\
\\
0 & \pmod{\hbar^2} & \text{otherwise},\\
\end{array}\right.\]
where $<$ is the Bruhat order on the Weyl group $W$.
\end{lem}

\section{$T$-equivariant quantum cohomology of $T^*\mathcal{P}$}
Now we turn to the study of equivariant quantum cohomology of $T^*\mathcal{P}$. We denote $T^*\mathcal{P}$ by $X$ in this section. Recall $D_{\lambda}:=c_1(\mathcal{L}_{\lambda})$. We are going to determine the quantum multiplication by the divisor $D_{\lambda}$ in terms of the stable basis. It is easy to see that $y\lambda$ does not depend on the choice of representative in $yW_P$, since $W_P$ fix $\lambda$.

\subsection{Preliminaries on quantum cohomology}

By definition, the operator of quantum multiplication by $\alpha\in H_T(X)$ has the following matrix elements
\[(\alpha\ast \gamma_1, \gamma_2)=\sum_{\beta\in H_2(X,\mathbb{Z})}q^{\beta} \langle \alpha,\gamma_1,\gamma_2 \rangle^X_{0,3,\beta},\]
where $(\cdot, \cdot)$ denotes the standard inner product on cohomology and the quantity in angle brackets is a 3-point, genus 0, degree $\beta$ equivariant Gromov--Witten invariant of $X$.

If $\alpha$ is a divisor and $\beta\neq 0$, we have 
\[\langle\alpha,\gamma_1,\gamma_2\rangle^X_{0,3,\beta}=(\alpha,\beta)\langle\gamma_1,\gamma_2\rangle^X_{0,2,\beta}.\]

Since X has a everywhere-nondegenerate holomorphic symplectic form, it is well-known that the usual non-equivariant virtual fundamental class on $\overline{M}_{g,n}(X,\beta)$ vanishes for $\beta\neq 0$. However, we can modify the standard obstruction theory so that the virtual dimension  increases by 1 (see \cite{Braverman2011} or \cite{Okounkov2004}). The virtual fundamental class $[\overline{M}_{0,2}(X,\beta)]^\text{vir}$ has expected dimension 
\[K_X\cdot \beta+ \dim X+2-3=\dim X-1.\]
Hence the reduced virtual class has dimension $\dim X$, and for any $\beta\neq 0$,
\[[\overline{M}_{0,2}(X,\beta)]^\text{vir}=-\hbar\cdot[\overline{M}_{0,2}(X,\beta)]^\text{red},\]
where $\hbar$ is the weight of the symplectic form under the $\mathbb{C}^*-$action.

\subsection{Unbroken curves}
Broken curves was introduced in \cite{Okounkov2004}. Let $f:C\rightarrow X$ be an $A$-fixed point of $\overline{M}_{0,2}(X,\beta)$ such that the domain is a chain of rational curves
\[C=C_1\cup C_2\cup \cdots \cup C_k,\]
with the marked points lying on $C_1$ and $C_k$ respectively.

We say $f$ is an unbroken chain if at every node $f(C_i\cap C_{i+1})$ of $C$, the weights of the two branches are opposite and nonzero. Note that all the nodes are fixed by $A$.

More generally, if $(C,f)$ is an $A$-fixed point of $\overline{M}_{0,2}(X,\beta)$, we say that $f$ is an unbroken map if it satisfies one of the three conditions:
\begin{enumerate}
\item
$f$ arises from a map $f:C\rightarrow X^A$,
\item
$f$ is an unbroken chain, or
\item 
the domain $C$ is a chain of rational curves
\[C=C_0\cup C_1\cup\cdots C_k\]
such that $C_0$ is contracted by $f$, the marked points lie on $C_0$, and the remaining components form an unbroken chain.
\end{enumerate}
Broken maps are $A$-fixed maps that do not satisfy any of these conditions.

Okounkov and Pandharipande proved the following Theorem in Section 3.8.3 in \cite{Okounkov2004}.
\begin{thm}[\cite{Okounkov2004}]\label{OP unbroken}
Every map in a given connected component  of $\overline{M}_{0,2}(X,\beta)^A$ is either broken or unbroken. Only unbroken components contribute to the $A$-equivariant localization of reduced virtual fundamental class. 
\end{thm}

\subsection{Unbroken curves in X}
Any $\alpha\in R^+\setminus R^+_P$ defines an $SL_2$ subgroup $G_{\alpha^{\vee}}$ of $G$ and hence a rational curve 
\[C_{\alpha}:=G_{\alpha^{\vee}}\cdot [P]\subset G/P\subset X.\] 
This is the unique $A$-invariant rational curve connecting the fixed points $\bar{1}$ and $\bar{\sigma}_{\alpha}$, because any such rational curve has tangent weight at $\bar{1}$ in $R^-\setminus R_P^-$, and uniqueness follows from the following lemma in Section 4 in \cite{Fulton2001}.

\begin{lem}[\cite{Fulton2001}]\label{fulton unique}
Let $\alpha, \beta$ be two roots in $R^+\setminus R^+_P$. Then $\bar{\sigma}_{\alpha}=\bar{\sigma}_{\beta}$ if and only if $\alpha=\beta$.
\end{lem}

If $C$ is an $A$-invariant rational curve in $X$, $C$ must lie in $G/P$, and it connects two fixed points $\bar{y}$ and $\bar{w}$. Then its $y^{-1}$-translate $y^{-1}C$ is still an $A$-invariant curve, which connects fixed points $\bar{1}$ and $\overline{y^{-1}w}$. So  $y^{-1}C=C_{\alpha}$ for a unique $\alpha\in R^+\setminus R^+_P$, and $\overline{y^{-1}w}=\bar{\sigma}_{\alpha}$. Hence the tangent weight of $C$ at $\bar{y}$ is $-y\alpha$.
In conclusion, we have

\begin{lem}\label{unbroken in P}
There are two kinds of unbroken curves $C$ in X:
\begin{enumerate}
\item $C$ is a multiple cover of rational curve branched over two different fixed points,
\item $C$ is a chain of two rational curve $C=C_0\cup C_1$, such that $C_0$ is contracted to a fixed point, the two marked points lie on $C_0$, and $C_1$ is a multiple cover of rational curve branched over two different fixed points.
\end{enumerate}
\end{lem}

For any $\alpha\in \Delta\setminus I$, define $\tau(\sigma_{\alpha}):=\overline{B\sigma_{\alpha}P/P}$. Then 
\[\{\tau(\sigma_{\alpha})|\alpha\in \Delta\setminus I\}\]
form a basis of $H_2(X,\mathbb{Z})$. Let $\{\omega_{\alpha}|\alpha\in \Delta\}$ be the fundamental weights of the root system. For any $\alpha\in R^+\setminus R^+_P$ , define degree $d(\alpha)$ of $\alpha$ by
\begin{equation}\label{degree}
d(\alpha)=\sum_{\beta\in \Delta\setminus I}(\omega_{\beta},\alpha^{\vee})\tau(\sigma_{\beta}).
\end{equation}
\begin{lem}[\cite{Fulton2001}]
The degree of $[C_{\alpha}]$ is $d(\alpha)$, and $d(\alpha)=d(w\alpha)$ for any $w\in W_P$.
\end{lem}

\subsection{Classical part}
We first calculate the classical multiplication by $D_{\lambda}$ in the stable basis. Let $m$ denote the dimension of $G/P$. Since 
$\{\stab_+({\bar{y}})\}$ and $\{(-1)^m\stab_-(\bar{y})\}$ are dual bases, we only need to calculate 
\begin{equation}\label{classical mult}
(D_{\lambda}\cup \stab_+({\bar{y}}), (-1)^m\stab_-(\bar{w}))=\sum\limits_{\bar{w}\leq\bar{z}\leq \bar{y}}\frac{D_{\lambda}|_{\bar{z}}\cdot \stab_+({\bar{y}})|_{\bar{z}}\cdot (-1)^m\stab_-(\bar{w})|_{\bar{z}}}{e(T_{\bar{z}}X)}.
\end{equation}
This will be zero if $\bar{y}<\bar{w}$. Assume $y$ is a minimal representative. Note that the resulting expression lies in the nonlocalized coefficient ring due to the proof of Theorem 4.4.1 in \cite{Maulik2012}, and a degree count shows that it is in $H_T^2(\text{pt})$. There are two cases.

\subsubsection{Case $\bar{y}=\bar{w}$}
There is only one term in the sum of the right hand side of Equation (\ref{classical mult}). Hence,
\[
(D_{\lambda}\cup \stab_+({\bar{y}}), (-1)^m\stab_-(\bar{y}))=\frac{D_{\lambda}|_{\bar{y}}\cdot \stab_+({\bar{y}})|_{\bar{y}}\cdot (-1)^m\stab_-(\bar{y})|_{\bar{y}}}{e(T_{\bar{y}}X)}=y(\lambda).
\]

\subsubsection{Case $\bar{y}\neq\bar{w}$}
Notice that $(D_{\lambda}\cup \stab_+({\bar{y}}), (-1)^m\stab_-(\bar{w}))\in H_T^2(\text{pt})$, and it is $0$ if $\hbar=0$, because every term in Equation (\ref{classical mult}) is divisible by $\hbar$. Hence, it is a constant multiple of $\hbar$. So in Equation (\ref{classical mult}), only $\bar{z}=\bar{y}$ and $\bar{z}=\bar{w}$ have contribution since all other terms are divisible by $\hbar^2$. Therefore,
\begin{align*}
(D_{\lambda}\cup \stab_+({\bar{y}}), (-1)^m\stab_-(\bar{w}))&=y(\lambda)\frac{\stab_-(\bar{w})|_{\bar{y}}}{\stab_-({\bar{y}})|_{\bar{y}}}+w(\lambda)\frac{\stab_+(\bar{y})|_{\bar{w}}}{\stab_+(\bar{w})|_{\bar{w}}}\\
&=y(\lambda)\frac{\hbar\text{ part of } \stab_-(\bar{w})|_{\bar{y}}}{\prod\limits_{\alpha\in R^+\setminus R^+_P}y\alpha}+w(\lambda)\frac{\hbar\text{ part of } \stab_+(\bar{y})|_{\bar{w}}}{\prod\limits_{\alpha\in R^+\setminus R^+_P}w\alpha},
\end{align*} 
where the first equality follows from $\stab_+({\bar{y}})\cdot \stab_-({\bar{y}}))=(-1)^me(T_{\bar{y}}X)$.

Lemma \ref{mod h^2 P} shows this is zero if $\bar{w}\neq \overline{y\sigma_{\beta}}$ for any $\beta\in R^+$ with $y\sigma_{\beta}<y$. However, if $\bar{w}=\overline{y\sigma_{\beta}}$ for such a $\beta$, then since $(-1)^{l(y\sigma_{\beta})}=(-1)^{l(y)+1}$, we have
\begin{align*}
& (D_{\lambda}\cup \stab_+({\bar{y}}), (-1)^m\stab_-(\bar{w}))\\
=&y(\lambda)(-1)^{l(y)+1}\frac{\hbar\prod\limits_{\alpha\in R^+}\alpha}{y\beta \prod\limits_{\alpha\in R^+}y\alpha}
+y\sigma_{\beta}(\lambda)(-1)^{l(y)+1}\frac{\hbar\prod\limits_{\alpha\in R^+}\alpha}{y\beta\prod\limits_{\alpha\in R^+} y\sigma_{\beta}\alpha}\\
=&-\frac{\hbar}{y\beta}y(\lambda)+\frac{\hbar}{y\beta}y\sigma_{\beta}(\lambda)\\
=&-\hbar(\lambda,\beta^{\vee}).
\end{align*} 

Notice that for any $\beta\in R^+$, $y\sigma_{\beta}<y$ is equivalent to $y\beta\in R^-$. To summarize, we get
\begin{thm}\label{classical part}
Let $y$ be a minimal representative. Then the classical multiplication is given by
\begin{align*}
D_{\lambda}\cup \stab_+({\bar{y}})
= y(\lambda)\stab_+({\bar{y}})-\hbar\sum\limits_{\alpha\in R^+, y\alpha\in R^-}(\lambda,\alpha^{\vee})\stab_+(\overline{y\sigma_{\alpha}}).
\end{align*}
\end{thm}

\subsection{Quantum part}
\label{quantum part}
Let $D_{\lambda}\ast_q$ denote the purely quantum multiplication. We want to calculate 
\[
(-1)^m(D_{\lambda}\ast_q \stab_+(\bar{y}), \stab_-(\bar{w}))=-\sum_{\beta\text{ effective}}(-1)^m\hbar q^{\beta}(D_{\lambda}, \beta)(ev_*[\overline{M}_{0,2}(X,\beta)]^\text{red},\stab_+(\bar{y})\otimes \stab_-(\bar{w})).
\]
where $ev$ is the evaluation map from $\overline{M}_{0,2}(X,\beta)$ to $X\times X$. The $-$ sign appears because the cotangent fibers have weight $-\hbar$ under the $\mathbb{C}^*-$action. Since 
\[\dim[\overline{M}_{0,2}(X, \beta)]^\text{red}=\dim X,\] and
\[
(ev_*[\overline{M}_{0,2}(X, \beta)]^\text{red},\stab_+(\bar{y})\otimes \stab_-(\bar{w}))
\]
lies in the nonlocalized coefficient ring (see Theorem 4.4.1 in \cite{Maulik2012}), the product is a constant by a degree count. Thus we can let $\hbar=0$, i.e., we can calculate it in $A$-equivariant chomology. As in the classical multiplication, there are two cases depending whether the two fixed points $\bar{y}$ and $\bar{w}$ are the same or not.

\subsubsection{Case $\bar{y}\neq\bar{w}$}
By virtual localization, Theorem \ref{OP unbroken} and Lemma \ref{unbroken in P}, 
\[
(ev_*[\overline{M}_{0,2}(X, \beta)]^\text{red},\stab_+(\bar{y})\otimes \stab_-(\bar{w}))
\]
is nonzero if and only if $\bar{w}=\overline{y\sigma_{\alpha}}$ for some $\alpha\in R^+\setminus R^+_P$. Only the first kind of unbroken curves have contribution to $(ev_*[\overline{M}_{0,2}(X, \beta)]^\text{red},\stab_+(\bar{y})\otimes \stab_-(\overline{y\sigma_{\alpha}}))$, and only restriction to the fixed point $(\bar{y}, \overline{y\sigma_{\alpha}})$ is nonzero in the localization of the product  by the first and third properties of the stable basis. The $A$-invariant rational curve $y[C_{\alpha}]$ connects the two fixed points $\bar{y}$ and $\overline{y\sigma_{\alpha}}$, and it is the unique one. For example, if $y[C_{\beta}]$ is also such a curve, then $\overline{y\sigma_{\alpha}}=\overline{y\sigma_{\beta}}=\bar{w}$. Hence $\alpha=\beta$ by Lemma \ref{fulton unique}. Therefore,
\begin{align*}
(-1)^m(D_{\lambda}\ast_q \stab_+(\bar{y}), \stab_-(\overline{y\sigma_{\alpha}}))=&-\sum_{k>0}(-1)^m\hbar q^{k\cdot d(\alpha)}(D_{\lambda}, k\cdot d(\alpha))\\
&(ev_*[\overline{M}_{0,2}(X,k\cdot d(\alpha))]^\text{red},\stab_+(\bar{y})\otimes \stab_-(\overline{y\sigma_{\alpha}})).
\end{align*}

Let $f$ be an unbroken map of degree $k$ from $C=\mathbb{P}^1$ to $y[C_{\alpha}]$. Then 
\[\Aut(f)=\mathbb{Z}/k.\]
By virtual localization,
\[
k(ev_*[\overline{M}_{0,2}(X,k\cdot d(\alpha))]^\text{red},\stab_+(\bar{y})\otimes \stab_-(\overline{y\sigma_{\alpha}}))=\frac{e(T_{\bar{y}}^*\mathcal{P})e(T_{\overline{y\sigma_{\alpha}}}^*\mathcal{P}) e'(H^1(C, f^*TX))}{e'(H^0(C, f^*TX))}.
\]
Here $e'$ is the product of nonzero $A$-weights.

We record Lemma 11.1.3 from \cite{Maulik2012}.
\begin{lem}[\cite{Maulik2012}]\label{line bundle on p1}
Let $A$ be a torus and let $\mathcal{T}$ be an $A$-equivariant bundle on $C = \mathbb{P}^1$ without zero weights in the fibers $\mathcal{T}_0$ and $\mathcal{T}_{\infty}$. Then 
\[\frac{e'(H^0(\mathcal{T}\oplus\mathcal{T}^*))}{e'(H^1(\mathcal{T}\oplus\mathcal{T}^*))}=(-1)^{\deg\mathcal{T}+rk\mathcal{T}+z}e(\mathcal{T}_0\oplus \mathcal{T}_{\infty})\]
where $z=\dim H^1(\mathcal{T}\oplus\mathcal{T}^*)^A$, i.e., $z$ counts the number of zero weights in $H^1(\mathcal{T}\oplus\mathcal{T}^*).$
\end{lem}
Since 
\[f^*TX=\mathcal{T}\oplus\mathcal{T}^* \text{\quad with\quad} \mathcal{T}=f^*T\mathcal{P},\]
Lemma \ref{line bundle on p1} gives
\begin{align*}
k(ev_*[\overline{M}_{0,2}(X,k\cdot d(\alpha))]^\text{red},\stab_+(\bar{y})\otimes \stab_-(\overline{y\sigma_{\alpha}}))
=&\frac{e(T_{\bar{y}}^*\mathcal{P})e(T_{\bar{y\sigma_{\alpha}}}^*\mathcal{P}) e'(H^1(C, f^*TX))}{e'(H^0(C, f^*TX))}\\
=&(-1)^{\deg\mathcal{T}+rk\mathcal{T}+z}.
\end{align*}
We now study the vector bundle $\mathcal{T}=f^*T\mathcal{P}$. First of all, $rk\mathcal{T}=\dim \mathcal{P}$. By localization,
\begin{align*}
\deg\mathcal{T}&=k\left(\frac{\sum\limits_{\gamma\in R^+\setminus R^+_P}(-y\gamma)}{-y\alpha}+\frac{\sum\limits_{\gamma\in R^+\setminus R^+_P}(-y\sigma_{\alpha}\gamma)}{y\alpha}\right)\\
&=k\sum\limits_{\gamma\in R^+\setminus R^+_P}(\gamma, \alpha^{\vee})=k(2\rho-2\rho_P,\alpha^{\vee})\\
&=2k\sum\limits_{\beta\in \Delta\setminus I}(\omega_{\beta}, \alpha^{\vee})
\end{align*}
is an even number, where $\rho$ is the half sum of the positive roots, $\rho_P$ is the half sum of the positive roots in $R_{P}^+$, and $\omega_{\beta}$ are the fundamental weights.

The vector bundle $\mathcal{T}$ splits as a direct sum of line bundles on $C$
\[\mathcal{T}=\bigoplus_{i}\mathcal{L}_i,\]
so 
\[\bigoplus_i\mathcal{L}_i|_0=\bigoplus_{\gamma\in R^+\setminus R^+_P}\mathfrak{g}_{-y\gamma},\]
where $\mathfrak{g}_{-y\gamma}$ are the root subspaces of $\mathfrak{g}$. Suppose $\mathcal{L}_i|_0=\mathfrak{g}_{-y\gamma}$. Since $y\sigma_{\alpha}y^{-1}$ maps $y$ to $y\sigma_{\alpha}$, we have 
\[\mathcal{L}_i|_{\infty}=\mathfrak{g}_{-y\sigma_{\alpha}\gamma}.\] 
Hence there is only one zero weight in $H^1(\mathcal{T}\oplus\mathcal{T}^*)$, which occurs in $H^1(\mathcal{L}_i\oplus\mathcal{L}_i^*)$, where $\mathcal{L}_i|_0=\mathfrak{g}_{-y\alpha}$, i.e., $\mathcal{L}_i$ is the tangent bundle of $C$.

Therefore $z=1$ and we have
\begin{lem}
\[
(-1)^m(D_{\lambda}\ast_q \stab_+({\bar{y}}), \stab_-({\overline{y\sigma_{\alpha}}}))=\sum_{k>0}\hbar q^{k\cdot d(\alpha)}(D_{\lambda}, d(\alpha))=-\hbar \frac{q^{d(\alpha)}}{1-q^{d(\alpha)}}(\lambda, \alpha^{\vee}).
\]
\end{lem}
\begin{proof}
We only need to show
\[(D_{\lambda}, d(\alpha))=-(\lambda, \alpha^{\vee}).\]
By definition and localization,
\begin{align*}
(D_{\lambda}, d(\alpha))&=\sum\limits_{\beta\in \Delta\setminus I}(\omega_{\beta,\alpha^{\vee}})\int_{\tau(\sigma_{\beta})}c_1(\mathcal{L}_{\lambda})=\sum\limits_{\beta\in \Delta\setminus I}(\omega_{\beta,\alpha^{\vee}})\left(\frac{\lambda}{-\beta}+\frac{\sigma_{\beta}\lambda}{\beta}\right)\\
&=-\sum\limits_{\beta\in \Delta\setminus I}(\omega_{\beta,\alpha^{\vee}})(\lambda,\beta^{\vee})=-\sum\limits_{\beta\in \Delta}(\omega_{\beta,\alpha^{\vee}})(\lambda,\beta^{\vee})\\
&=-(\lambda,\alpha^{\vee}).
\end{align*}
\end{proof}
\subsubsection{Case $\bar{y}=\bar{w}$}
In this case, only the second kind of unbroken curves have contribution to $(D_{\lambda}\ast_q \stab_+(\bar{y}), \stab_-(\bar{y}))$. Let $C=C_0\cup C_1$ be an unbroken curve of the second kind with $C_0$ contracted to the fixed point $\bar{y}$, and $C_1$ is a cover of the rational curve $yC_{\alpha}$ of degree $k$, where $\alpha\in R^+\setminus R^+_P$. Let $p$ denote the node of $C$, and let $f$ be the map from $C$ to $X$. Then the corresponding decorated graph $\Gamma$ has two vertices, one of them has two marked tails, and there is an edge of degree $k$ connecting the two vertices. Hence the automorphism group of the graph is trivial. The virtual normal bundle (\cite{Mirror}) is 
\begin{equation}\label{Normal bundle}
e(N_{\Gamma}^{\vir})=\frac{e'(H^0(C,f^*TX))}{e'(H^1(C,f^*TX))} \frac{-y\alpha/k}{y\alpha/k}\\
=-\frac{e'(H^0(C,f^*TX))}{e'(H^1(C,f^*TX))},
\end{equation} 
where $e'(H^0(C,f^*TX))$ denotes the nonzero $A$-weights in $H^0(C,f^*TX)$. Consider the normalization exact sequence resolving the node of $C$:
\[0\rightarrow \mathcal{O}_{C}\rightarrow \mathcal{O}_{C_0}\oplus \mathcal{O}_{C_1}\rightarrow \mathcal{O}_p\rightarrow 0.\]
Tensoring with $f^*TX$ and taking cohomology yields:
\begin{align*}
0\rightarrow H^0(C,f^*TX)&\rightarrow H^0(C_0,f^*TX)\oplus H^0(C_1,f^*TX)\rightarrow T_{\bar{y}}X\\
&\rightarrow H^1(C,f^*TX)\rightarrow H^1(C_0,f^*TX)\oplus H^1(C_1,f^*TX)\rightarrow 0.
\end{align*}
Since $C_0$ is contracted to $\bar{y}$, $H^0(C_0,f^*TX)=T_{\bar{y}}X$ and $H^1(C_0,f^*TX)=0$. Therefore, as virtual representations, we have
\[H^0(C,f^*TX)-H^1(C,f^*TX)=H^0(C_1,f^*TX)-H^1(C_1,f^*TX).\]
Due to Equation (\ref{Normal bundle}) and the analysis in the last case, we get
\begin{align*}
e(N_{\Gamma}^{\vir})&=-\frac{e'(H^0(C_1,f^*TX))}{e'(H^1(C_1,f^*TX))}\\
&=(-1)^me(T_{\bar{y}}\mathcal{P})e(T_{\overline{y\sigma_{\alpha}}}\mathcal{P}).
\end{align*}
Then by virtual localization formula, we have
\begin{align*}
(-1)^m(D_{\lambda}\ast_q \stab_+(\bar{y}), \stab_-(\bar{y}))&=-\hbar\sum_{\alpha\in R^+\setminus R^+_P, k>0}(D_{\lambda}, d(\alpha))q^{k\cdot d(\alpha)}\frac{e(T_{\bar{y}}^*\mathcal{P})^2}{e(T_{\bar{y}}\mathcal{P})e(T_{\overline{y\sigma_{\alpha}}}\mathcal{P})}\\
&=\hbar\sum_{\alpha\in R^+\setminus R^+_P}(\lambda, \alpha^{\vee})\frac{q^{d(\alpha)}}{1-q^{d(\alpha)}}\frac{\prod\limits_{\beta\in R^+\setminus R^+_P}y\beta}{\prod\limits_{\beta\in R^+\setminus R^+_P}y\sigma_{\alpha}\beta}\\
&=\hbar\sum_{\alpha\in R^+\setminus R^+_P}(\lambda, \alpha^{\vee})\frac{q^{d(\alpha)}}{1-q^{d(\alpha)}}\frac{\prod\limits_{\beta\in R^+}y\beta}{\prod\limits_{\beta\in R^+}y\sigma_{\alpha}\beta}\frac{\prod\limits_{\beta\in R^+_P}y\sigma_{\alpha}\beta}{\prod\limits_{\beta\in R^+_P}y\beta}\\
&=-\hbar\cdot y\left(\sum_{\alpha\in R^+\setminus R^+_P}(\lambda, \alpha^{\vee})\frac{q^{d(\alpha)}}{1-q^{d(\alpha)}}\frac{\prod\limits_{\beta\in R^+_P}\sigma_{\alpha}\beta}{\prod\limits_{\beta\in R^+_P}\beta}\right).
\end{align*}
Here we have used \[\prod\limits_{\beta\in R^+}y\beta=(-1)^{l(y)}\prod\limits_{\beta\in R^+}\beta, \text{\quad and\quad } (-1)^{l(y\sigma_{\alpha})}=(-1)^{l(y)+l(\sigma_{\alpha})}=(-1)^{l(y)+1}.\] 
Notice that for any root $\gamma\in R^+_P$, $\sigma_{\gamma}$ preserves $R^+\setminus R^+_P$. For any $\alpha\in R^+\setminus R^+_P$, $d(\sigma_{\gamma}(\alpha))=d(\alpha)$, $(\lambda, \alpha^{\vee})=(\lambda, \sigma_{\gamma}(\alpha)^{\vee})$ and $\prod\limits_{\beta\in R^+_P}\sigma_{\gamma}\beta=-\prod\limits_{\beta\in R^+_P}\beta$. Hence,
\begin{align*}
\sigma_{\gamma}&\left(\sum_{\alpha\in R^+\setminus R^+_P}(\lambda, \alpha^{\vee})\frac{q^{d(\alpha)}}{1-q^{d(\alpha)}}\prod\limits_{\beta\in R^+_P}\sigma_{\alpha}\beta\right)\\
&=\sum_{\alpha\in R^+\setminus R^+_P}(\lambda, \alpha^{\vee})\frac{q^{d(\alpha)}}{1-q^{d(\alpha)}}\prod\limits_{\beta\in R^+_P}\sigma_{\sigma_{\gamma}\alpha}\sigma_{\gamma}\beta\\
&=-\sum_{\alpha\in R^+\setminus R^+_P}(\lambda, \alpha^{\vee})\frac{q^{d(\alpha)}}{1-q^{d(\alpha)}}\prod\limits_{\beta\in R^+_P}\sigma_{\alpha}\beta.
\end{align*}
Therefore $\sum\limits_{\alpha\in R^+\setminus R^+_P}(\lambda, \alpha^{\vee})\frac{q^{d(\alpha)}}{1-q^{d(\alpha)}}\prod\limits_{\beta\in R^+_P}\sigma_{\alpha}\beta$ is divisible by $\prod\limits_{\beta\in R^+_P}\beta$. But they have the same degree, so 
\begin{equation}\label{constant}
\sum_{\alpha\in R^+\setminus R^+_P}(\lambda, \alpha^{\vee})\frac{q^{d(\alpha)}}{1-q^{d(\alpha)}}\frac{\prod\limits_{\beta\in R^+_P}\sigma_{\alpha}\beta}{\prod\limits_{\beta\in R^+_P}\beta}
\end{equation}
is a scalar. 

To summarize, we get
\begin{thm}\label{purely quantum mul for P}
The purely quantum multiplication by $D_{\lambda}$ in $H_T^*(T^*\mathcal{P})$ is given by:
\begin{align*}
D_{\lambda}\ast_q \stab_+({\bar{y}})=-\hbar\sum_{\alpha\in R^+\setminus R^+_P}(\lambda, \alpha^{\vee})\frac{q^{d(\alpha)}}{1-q^{d(\alpha)}} \stab_+(\overline{y\sigma_{\alpha}})
-\hbar\sum_{\alpha\in R^+\setminus R^+_P}(\lambda, \alpha^{\vee})\frac{q^{d(\alpha)}}{1-q^{d(\alpha)}}\frac{\prod\limits_{\beta\in R^+_P}\sigma_{\alpha}\beta}{\prod\limits_{\beta\in R^+_P}\beta} \stab_+(\bar{y}).
\end{align*}
\end{thm}
\begin{rem}
\leavevmode
\begin{enumerate}

\item
The scalar 
\[-\hbar\sum_{\alpha\in R^+\setminus R^+_P}(\lambda, \alpha^{\vee})\frac{q^{d(\alpha)}}{1-q^{d(\alpha)}}\frac{\prod\limits_{\beta\in R^+_P}\sigma_{\alpha}\beta}{\prod\limits_{\beta\in R^+_P}\beta}\]
can also be determined by the condition 
\[D_{\lambda}\ast_q1=0.\]
\item
The element $y$ is not necessarily a minimal representative.
\item 
The Theorem is also true if we replace all the $\stab_+$ by $\stab_-$.
\end{enumerate}
\end{rem}

\subsection{Quantum multiplications}
Combining Theorem \ref{classical part} and Theorem \ref{purely quantum mul for P}, we get our main Theorem \ref{quantum mul for P}. Taking $I=\emptyset$, we get the quantum multiplication by $D_{\lambda}$ in $H_T^*(T^*\mathcal{B})$.
\begin{thm}\label{quantum mul for B}
The quantum multiplication by $D_{\lambda}$ in $H_T^*(T^*\mathcal{B})$ is given by:
\begin{align*}
D_{\lambda}\ast \stab_+(y)=& y(\lambda)\stab_+(y)-\hbar\sum\limits_{\alpha\in R^+, y\alpha\in -R^+}(\lambda,\alpha^{\vee})\stab_+(y\sigma_{\alpha})\\
&-\hbar\sum_{\alpha\in R^+}(\lambda, \alpha^{\vee}) \frac{q^{\alpha^{\vee}}}{1-q^{\alpha^{\vee}}}(\stab_+(y\sigma_{\alpha})+\stab_+(y)).
\end{align*}
\end{thm}

\subsection{Calculation of the scalar in type A}
\label{Calculation of constants}
We can define an equivalence relation on $R^+\setminus R^+_P$ as follows
\[\alpha\sim \beta \text{\quad if \quad} d(\alpha)=d(\beta).\]
Then $w(\alpha)\sim \alpha$ for any $w\in W_P$. We have
\begin{align*}
&\sum\limits_{\alpha\in R^+\setminus R^+_P}(\lambda, \alpha^{\vee})\frac{q^{d(\alpha)}}{1-q^{d(\alpha)}}\frac{\prod\limits_{\beta\in R^+_P}\sigma_{\alpha}\beta}{\prod\limits_{\beta\in R^+_P}\beta}\\
&=\sum\limits_{\alpha\in (R^+\setminus R^+_P)/\sim}(\lambda, \alpha^{\vee})\frac{q^{d(\alpha)}}{1-q^{d(\alpha)}}\sum\limits_{\alpha'\sim\alpha}\frac{\prod\limits_{\beta\in R^+_P}\sigma_{\alpha'}\beta}{\prod\limits_{\beta\in R^+_P}\beta}.
\end{align*}
It is easy to see that 
\[\sum\limits_{\alpha'\sim\alpha}\frac{\prod\limits_{\beta\in R^+_P}\sigma_{\alpha'}\beta}{\prod\limits_{\beta\in R^+_P}\beta}\]
is a constant, which will be denoted by $C_P(\alpha)$. 

In this section, we will determine the constant $C_P(\alpha)$ when $G$ is of type $A$. We will first calculate this number in $T^*Gr(k,n)$ case, and the general case will follow easily. Now let $G=SL(n, \mathbb{C})$ and let $x_i$ be the function on the Lie algebra of the diagonal torus defined by $x_i(t_1,\cdots,t_n)=x_i$. 

\subsubsection{$T^*Gr(k,n)$ case}
Let $P$ be a parabolic subgroup containing the upper triangular matrices such that $T^*(G/P)$ is $T^*Gr(k,n)$. Then 
\[R_P^+=\{x_i-x_j|1\leq i<j \leq  k, \text{ or } k<i<j\leq n\},\quad R\setminus R_P^+=\{x_i-x_j|1\leq i\leq k<j\leq n\}\]
and all the roots in $R\setminus R_P^+$ are equivalent. The number $C_P(\alpha)$ will be denoted by $C_P$. By definition, 
\begin{equation}\label{CP}
C_P=\frac{\sum\limits_{1\leq r\leq k<s\leq n}(rs)\left(\prod\limits_{1\leq i< j\leq k}(x_i-x_j)\prod\limits_{1+k\leq p< q\leq n}(x_p-x_q)\right)}{\prod\limits_{1\leq i< j\leq k}(x_i-x_j)\prod\limits_{1+k\leq p< q\leq n}(x_p-x_q)},
\end{equation}
where $(rs)$ means the transposition of $x_r$ and $x_s$. 

Observe that 
\[\prod\limits_{1\leq i< j\leq k}(x_j-x_i)\prod\limits_{1+k\leq p< q\leq n}(x_q-x_p)=\det\begin{pmatrix}
1&x_1&\cdots&x_1^{k-1}&&\\
\vdots&\vdots&\ddots&\vdots&&\\
1&x_k&\cdots&x_k^{k-1}&&\\
&&&& 1&x_{k+1}&\cdots&x_{k+1}^{n-k-1}&&\\
&&&& \vdots&\vdots&\ddots&&\\
&&&& 1&x_{n}&\cdots&x_{n}^{n-k-1}&&\\
\end{pmatrix}.\]
Then it is easy to see that the coefficient of $x_2x_3^2\cdots x_k^{k-1}x_{k+2}x_{k+3}^2\cdots x_n^{n-k-1}$ in 
\[\sum\limits_{1\leq r\leq k<s\leq n}(rs)\left(\prod\limits_{1\leq i< j\leq k}(x_j-x_i)\prod\limits_{1+k\leq p< q\leq n}(x_q-x_p)\right)\] is $\min(k,n-k)$, since only when $s-r=k$, $(rs)\left(\prod\limits_{1\leq i< j\leq k}(x_j-x_i)\prod\limits_{1+k\leq p< q\leq n}(x_q-x_p)\right)$ has the term $x_2x_3^2\cdots x_k^{k-1}x_{k+2}x_{k+3}^2x_n^{n-k-1}$, and the coefficient is $1$.
Hence 
\begin{prop}
\[C_P=\min(k,n-k).\]
\end{prop}

\subsubsection{General case}
Let $\lambda=(\lambda_1,\cdots,\lambda_N)$ be a partition of $n$ with $\lambda_1\geq\cdots\geq\lambda_N$. Let
\[\mathcal{F}_{\lambda}=\{0\subset V_1\subset V_2\cdots \subset V_N|\dim V_i/V_{i-1}=\lambda_i\}\]
be the partial flag variety, and let $P$ be the corresponding parabolic subgroup. Then 
\[R_P^+=\{x_i-x_j|\lambda_1+\cdots+\lambda_p< i<j \leq  \lambda_1+\cdots+\lambda_{p+1}, \text{ for some } p \text{ between } 0 \text{ and } N-1 \}.\] 
Two positive roots $x_i-x_j$ and $x_k-x_l$ are equivalent if and only if there exist $1\leq p< q\leq N$ such that 
\[\lambda_1+\cdots+\lambda_p< i,k \leq  \lambda_1+\cdots+\lambda_{p+1},  \lambda_1+\cdots+\lambda_q< j,l \leq  \lambda_1+\cdots+\lambda_{q+1}.\]
So the set $(R^+\setminus R^+_P)/\sim$ has representatives 
\[\{x_{\lambda_1+\cdots+\lambda_p}-x_{\lambda_1+\cdots+\lambda_{q}}|1\leq p< q\leq N\}.\]
The same analysis as in the last case gives
\begin{prop}
For any $1\leq p< q\leq N$,
\[C_P(x_{\lambda_1+\cdots+\lambda_p}-x_{\lambda_1+\cdots+\lambda_q})=\lambda_q.\]
\end{prop}

\section{$G\times\mathbb{C}^*$ quantum multiplications}
Let $\mathbb{G}=G\times \mathbb{C}^*$, and let $\mathcal{B}$ denote the flag variety $G/B$. In this section, we will first get the $\mathbb{G}$-equivariant quantum multiplication formula in $T^*\mathcal{B}$, which is the main result of \cite{Braverman2011}. Then we show the quantum multiplication formula in $T^*\mathcal{P}$ is conjugate to the conjectured formula given by Braverman .

\subsection{$T^*\mathcal{B}$ case}
Let us recall the result from \cite{Braverman2011} first. Let $\mathfrak{t}$ be the Lie algebra of the maximal torus $A$. Then
\[H_\mathbb{G}^*(T^*\mathcal{B})\simeq H_T^*(T^*\mathcal{B})^W\simeq H_T^*(\text{pt})\simeq \sym\mathfrak{t}^*[\hbar].\]
The isomorphism is determined as follows: for any $\beta\in H_\mathbb{G}^*(T^*\mathcal{B})$, lift it to $H_T^*(T^*\mathcal{B})$, and then restrict it to the fixed point $1$. Similarly, we have
\[H_\mathbb{G}^*(T^*\mathcal{P})\simeq H_T^*(T^*\mathcal{P})^W\simeq (\sym\mathfrak{t}^*)^{W_P}[\hbar].\]

Let us recall the definition of the graded affine Hecke algebra $\mathcal{H}_{\hbar}$. It is generated by the symbols $x_{\lambda}$ for $\lambda\in \mathfrak{t}^*$, Weyl elements $\bar{w}$ and a central element $\hbar$ such that
\begin{enumerate}
\item
$x_{\lambda}$ depends linearly on $\lambda\in \mathfrak{t}^*$;
\item
$x_{\lambda}x_{\mu}=x_{\mu}x_{\lambda}$;
\item
the $\tilde{w}$'s form the Weyl group inside $\mathcal{H}_t$;
\item
for any $\alpha\in \Delta$, $\lambda\in \mathfrak{t}^*$, we have
\[\tilde{\sigma}_{\alpha}x_{\lambda}-x_{\tilde{\sigma}_{\alpha}(\lambda)}\tilde{\sigma}_{\alpha}=\hbar(\alpha^{\vee},\lambda).\]
\end{enumerate}
According to \cite{lusztig1988cuspidal}, we have a natural isomorphism
\[H_*^\mathbb{G}(T^*\mathcal{B}\times_{\mathcal{N}}T^*\mathcal{B})\simeq \mathcal{H}_{\hbar},\]
where $\mathcal{N}$ is the nilpotent cone in $\mathfrak{g}$.  The action of $\mathcal{H}_{\hbar}$ on $\sym\mathfrak{t}^*[\hbar]$ is defined as follows: $x_{\lambda}$ acts by multiplication by $\lambda$, and for every simple root $\alpha$, the action of $\tilde{\sigma}_{\alpha}$ is defined by
\[
\tilde{\sigma}_{\alpha}f=(\frac{\hbar}{\alpha }+\frac{\alpha-\hbar}{\alpha}\sigma_{\alpha})f
\]
where $f\in \sym\mathfrak{t}^*[\hbar]$, and $\sigma_{\alpha}f$ is the usual Weyl group action on $\sym\mathfrak{t}^*[\hbar]$.

Having introduced the above notations, we can state the main Theorem of \cite{Braverman2011}.
\begin{thm}[\cite{Braverman2011}]\label{G-equivariant quantum mul B}
The operator of quantum multiplication by $D_{\lambda}$ in $H_\mathbb{G}^*(T^*\mathcal{B})$ is equal to 
\[x_{\lambda}+\hbar\sum_{\alpha\in R^+}(\lambda,\alpha^{\vee})\frac{q^{\alpha^{\vee}}}{1-q^{\alpha^{\vee}}}(\tilde{\sigma}_{\alpha}-1).\]
\end{thm}

Let us also recall the restriction formula for stable basis from \cite{su2015restriction}.
\begin{thm}\label{restriction B}
Let $y=\sigma_1\sigma_2\cdots\sigma_l$ be a reduced expression for $y\in W$, and $w\leq y$. Then
\[
\stab_+(y)|_w=\sum\limits_{\substack{1\leq i_1<i_2<\dots<i_k\leq l\\ 
w=\sigma_{i_1}\sigma_{i_2}\dots\sigma_{i_k}}}(-1)^l \prod\limits_{j=1}^k\frac{\sigma_{i_1}\sigma_{i_2}\dots\sigma_{i_j}\alpha_{i_j}-\hbar}{\sigma_{i_1}\sigma_{i_2}\dots\sigma_{i_j}\alpha_{i_j}} \frac{\hbar^{l-k}}{\prod\limits_{j=0}^k\prod\limits_{i_j<r<i_{j+1}}\sigma_{i_1}\sigma_{i_2}\dots\sigma_{i_j}\alpha_r}\prod\limits_{\alpha\in R^+}\alpha,
\]
where $\sigma_i$ is the simple reflection associated to a simple root $\alpha_i$.
\end{thm}

We are now ready to deduce Theorem \ref{G-equivariant quantum mul B} from Theorem \ref{quantum mul for B} and Theorem \ref{restriction B}. The classical multiplication is obvious. We only show that the purely quantum part matches. Let $f\in \sym\mathfrak{t}^*[\hbar]$ correspond to $\gamma\in H_\mathbb{G}^*(T^*\mathcal{B})$. We also let $\gamma$ denote the lift in $H_T^*(T^*\mathcal{B})$. Then $\gamma|_w=w(f)$ for any $w\in W$. Since the stable and unstable basis are dual basis up to $(-1)^n$, where $n=\dim \mathcal{B}$, we have
\[\gamma=\sum_{y}(-1)^n(\gamma, \stab_+(y))\stab_-(y).\]
Due to Theorem \ref{quantum mul for B}, we have
\[
D_{\lambda}\ast_q \gamma=-\hbar\sum_{\alpha\in R^+}(\lambda, \alpha^{\vee})\frac{q^{\alpha^{\vee}}}{1-q^{\alpha^{\vee}}}\sum_{y}(\gamma, (-1)^n \stab_+(y))(\stab_-(y\sigma_{\alpha})+\stab_-(y)).
\]
Notice that $\stab_-(y)|_1=\delta_{y,1}e(T_1^*\mathcal{B})$. Restricting to the fixed point $1$ , we get
\begin{align*}
D_{\lambda}\ast_q \gamma|_1 &=-\hbar\sum_{\alpha\in R^+}(\lambda, \alpha^{\vee})\frac{q^{\alpha^{\vee}}}{1-q^{\alpha^{\vee}}}\gamma|_1\\
&-\hbar\sum_{\alpha\in R^+}(\lambda, \alpha^{\vee})\frac{q^{\alpha^{\vee}}}{1-q^{\alpha^{\vee}}}(\gamma, (-1)^n \stab_+(\sigma_{\alpha}))e(T_1^*\mathcal{B}).
\end{align*}
Hence we only need to show 
\begin{equation}\label{B}
-(\gamma, (-1)^n \stab_+(\sigma_{\alpha}))e(T_1^*\mathcal{B})=\tilde{\sigma}_{\alpha}f.
\end{equation}
To prove this, we need the following lemma.
\begin{lem}\label{cotangent}
If $w=\sigma_{i_1}\sigma_{i_2}\dots\sigma_{i_k}$, then
\[\prod_{j=1}^k\frac{\sigma_{i_1}\sigma_{i_2}\dots\sigma_{i_{j-1}}\alpha_{i_j}-\hbar}{\sigma_{i_1}\sigma_{i_2}\dots\sigma_{i_j}\alpha_{i_j}-\hbar}=\frac{e(T_1^*\mathcal{B})}{e(T^*_w\mathcal{B})}.\]
\end{lem}
\begin{proof}
If $w=\sigma_{i_1}\sigma_{i_2}\dots\sigma_{i_k}$ is reduced, then this follows from the fact
\[\{w\beta|\beta\in R^+, w\beta\in R^-\}=\{\sigma_{i_1}\sigma_{i_2}\dots\sigma_{i_{j}}\alpha_{i_j}|1\leq j\leq l\}.\]
If $w=(\sigma_{\alpha}\sigma_{\beta})^{m(\alpha,\beta)}=1$ for some simple roots $\alpha$ and $\beta$, where $m(\alpha,\beta)$ is the order of $\sigma_{\alpha}\sigma_{\beta}$, we can check it case by case easily. If $w=\sigma_{\alpha}^2$, then it it trivial. In general, $w$ will be a composition of these three cases.
\end{proof}
If $\sigma_{\alpha}=\sigma_{\alpha_1}\cdots\sigma_{\alpha_l}$ is a reduced decomposition, then 
\[\tilde{\sigma}_{\alpha}f=\prod_{i=1}^l(\frac{\hbar}{\alpha_i}+\frac{\alpha_i-\hbar}{\alpha_i}\sigma_{\alpha_i})f.\]
Expanding this and using Theorem \ref{restriction B}, Lemma \ref{cotangent} and the fact $(-1)^{l(\sigma_{\alpha})}=-1$, we get
\begin{equation}\label{simple reflection action}
\tilde{\sigma}_{\alpha}(f)=\sum_{w}\frac{\stab_+(\sigma_{\alpha})|_w wf}{e(T_wT^*\mathcal{B})}(-1)^{1+n}e(T_1^*\mathcal{B})
=-(\gamma, (-1)^{n}\stab_+(\sigma_{\alpha}))e(T_1^*\mathcal{B}),
\end{equation}
which is precisely Equation (\ref{B}).

\subsection{$T^*\mathcal{P}$ case}
In the parabolic case, Professor Braverman suggests (through private communication) that the quantum multiplication should be 
\begin{equation}\label{conj}
D_{\lambda}\ast=x_{\lambda}+\hbar\sum_{\alpha\in R^+\setminus R^+_P}(\lambda,\alpha^{\vee})\frac{q^{d(\alpha)}}{1-q^{d(\alpha)}}\tilde{\sigma}_{\alpha}+\cdots,
\end{equation}
where $\cdots$ is some scalar. Recall we have 
\[H_\mathbb{G}^*(T^*\mathcal{P})\simeq H_T^*(T^*\mathcal{P})^W\simeq (\sym\mathfrak{t}^*)^{W_P}[\hbar].\]
It is easy to see that classical multiplication by $D_{\lambda}$ is given by multiplication by $\lambda$.

Now we do the similar calculation as in the $T^*\mathcal{B}$ case. We need the following restriction formula from \cite{su2015restriction}:
\begin{equation}\label{restriction formula for P}
\stab_{\pm}(\bar{y})|_{\bar{w}}=\sum_{\bar{z}=\bar{w}}\frac{\stab_\pm(y)|_z}{\prod\limits_{\alpha\in R^+_P}z\alpha}.
\end{equation}
Take any $\gamma\in H_\mathbb{G}^*(T^*\mathcal{P})$, and assume it corresponds to $f\in (\sym\mathfrak{t}^*)^{W_P}[\hbar]$. We still let $\gamma$ denote the corresponding lift in $H_T^*(T^*\mathcal{P})$. Then $\gamma|_{\bar{y}}=yf$. Let $m$ be the dimension of $\mathcal{P}$. Then we have
\[\gamma=\sum_{\bar{y}}(-1)^m(\gamma, \stab_+(\bar{y}))\stab_-(\bar{y}).\]
By Theorem \ref{purely quantum mul for P},
\begin{align*}
D_{\lambda}\ast_q \gamma&=\sum_{\bar{y}}(\gamma, (-1)^m \stab_+(\bar{y})) (-\hbar)\sum_{\alpha\in R^+\setminus R^+_P}(\lambda, \alpha^{\vee})\frac{q^{d(\alpha)}}{1-q^{d(\alpha)}} \stab_-(\overline{y\sigma_{\alpha}})\\
&-\hbar\sum_{\alpha\in R^+\setminus R^+_P}(\lambda, \alpha^{\vee})\frac{q^{d(\alpha)}}{1-q^{d(\alpha)}}\frac{\prod\limits_{\beta\in R^+_P}\sigma_{\alpha}\beta}{\prod\limits_{\beta\in R^+_P}\beta} \gamma.
\end{align*}
Notice that 
\[\stab_-(\overline{y\sigma_{\alpha}})|_{\bar{1}}
=\left\{\begin{array}{cc}
e(T_{\bar{1}}^*\mathcal{P}) & \text{ if\quad } \overline{y\sigma_{\alpha}}=\bar{1} ;\\
0& \text{ otherwise }.
\end{array}\right.
\]
Restricting $D_{\lambda}\ast_q \gamma$ to the fixed point $\bar{1}$, we get
\begin{align*}
D_{\lambda}\ast_q \gamma|_{\bar{1}}&=-\hbar\sum_{\alpha\in R^+\setminus R^+_P}(\lambda, \alpha^{\vee})\frac{q^{d(\alpha)}}{1-q^{d(\alpha)}}(\gamma, (-1)^m \stab_+(\bar{\sigma}_{\alpha}))e(T_{\bar{1}}^*\mathcal{P})\\
&-\hbar\sum_{\alpha\in R^+\setminus R^+_P}(\lambda, \alpha^{\vee})\frac{q^{d(\alpha)}}{1-q^{d(\alpha)}}\frac{\prod\limits_{\beta\in R^+_P}\sigma_{\alpha}\beta}{\prod\limits_{\beta\in R^+_P}\beta}f 
\end{align*}
Due to restriction formula (\ref{restriction formula for P}) and Equation (\ref{simple reflection action}), we have
\[(\gamma, (-1)^m \stab_+(\bar{\sigma}_{\alpha}))e(T_{\bar{1}}^*\mathcal{P})=-\frac{\tilde{\sigma}_{\alpha}(f\prod\limits_{\beta\in R^+_P}(\beta-\hbar))}{\prod\limits_{\beta\in R^+_P}(\beta-\hbar)}.\]
Hence, we obtain Theorem \ref{P con}. 

Since
\begin{equation}\label{constant term}
\hbar\sum_{\alpha\in R^+\setminus R^+_P}(\lambda, \alpha^{\vee})\frac{q^{d(\alpha)}}{1-q^{d(\alpha)}}\frac{\prod\limits_{\beta\in R^+_P}\sigma_{\alpha}\beta}{\prod\limits_{\beta\in R^+_P}\beta}
\end{equation}
is a scalar, the quantum multiplication formula in Theorem \ref{P con} is conjugate to the conjectured formula (\ref{conj}) by the function 
\[\prod\limits_{\beta\in R^+_P}(\beta-\hbar).\]
This factor comes from geometry as follows. Let $\pi$ be the projection map from $\mathcal{B}$ to $\mathcal{P}$, and $\Gamma_{\pi}$ be its graph. Then the conormal bundle to $\Gamma_{\pi}$ in $\mathcal{B}\times \mathcal{P}$ is a Lagrangian submanifold of $T^*(\mathcal{B}\times \mathcal{P})$. 
\[\xymatrix{
T_{\Gamma_{\pi}}^*(\mathcal{B}\times \mathcal{P}) \ar[r]^-{p_1} \ar[d]_{p_2} & T^*\mathcal{B}  \\
T^*\mathcal{P} \\}. \]
Let $D=p_{1*}p_2^*$ be the map from $H_{\mathbb{G}}^*(T^*\mathcal{P})$ to $H_{\mathbb{G}}^*(T^*\mathcal{B})$ induced by this correspondence. Then under the isomorphisms 
\[H_{\mathbb{G}}^*(T^*\mathcal{B})\simeq \sym\mathfrak{t}^*[\hbar] \text{\quad and \quad} H_\mathbb{G}^*(T^*\mathcal{P})\simeq (\sym\mathfrak{t}^*)^{W_P}[\hbar],\]
the map becomes multiplicaiton by the above factor, see \cite{su2015restriction}. The scalar in the conjectured formula (\ref{conj}) is just the one in Equation (\ref{constant term}). By the calculation in the Subsection \ref{Calculation of constants}, it is not equal to 
\[\hbar\sum\limits_{\alpha\in R^+\setminus R^+_P}(\lambda, \alpha^{\vee})\frac{q^{d(\alpha)}}{1-q^{d(\alpha)}}\]
in general. It can also be determined by the condition $D_{\lambda}\ast_q1=0$.

\bibliographystyle{plain}
\bibliography{quantum}

\end{document}